\documentclass{amsart}
\pdfoutput=1

\usepackage[utf8]{inputenc}
\usepackage{amsfonts,mathtools,amsmath,amsthm,amssymb,tikz,color,enumitem,latexsym,tikz-cd,mathrsfs,graphicx,verbatim,MnSymbol}
\usepackage[breaklinks]{hyperref}

\usepackage[backref=true]{biblatex}
\bibliography{sources.bib}
\allowdisplaybreaks

\hypersetup{
    colorlinks,
    hyperfootnotes=true,
    linkcolor={red!50!black},
    citecolor={blue!50!black},
    urlcolor={blue!80!black}
}

\DefineBibliographyStrings{english}{%
	backrefpage = {page},%
	backrefpages = {pages},%
}

\usetikzlibrary{arrows}

\tikzset{slopearrow/.style={sloped, anchor=south}}
  
\newtheorem{lem}{Lemma}[section]

\newtheorem{thm}[lem]{Theorem}
\newtheorem{prop}[lem]{Proposition}

\newtheorem{cor}[lem]{Corollary}
\newtheorem{defn}[lem]{Definition}

\theoremstyle{definition}

\newtheorem{rmk}[lem]{Remark}

\def\CC{\mathbb C}

\def\QQ{\mathbb Q}
\def\PP{\mathbb P}

\def\bbA{\mathbb A}

\def\scO{\mathscr O}

\def\scP{\mathscr P}
\def\scE{\mathscr E}
\def\scF{\mathscr F}
\def\scZ{\mathcal Z}
\def\caP{\mathcal P}
\def\caQ{\mathcal Q}

\def\caT{\mathcal T}

\def\zaropen{\mathcal{U}}
\def\numfield{k}

\DeclareMathOperator\N{N}

\DeclareMathOperator\CH{CH}

\DeclareMathOperator\rank{rank}
\DeclareMathOperator\Gal{Gal}

\DeclareMathOperator\chr{char}
\DeclareMathOperator\Sym{Sym}

\def\S{S}
\def\s{s}
\def\T{T}
\def\t{t}
\def\r{r}
\def\one{_1}
\def\two{_2}
\def\badred{R}
\DeclareMathOperator\val{val}

\title{Rational Equivalences on Products of Elliptic Curves in a Family}
\author{Jonathan Love}
\address{Stanford University, Dept.\ of Mathematics}
\email{jonlove@stanford.edu}
\thanks{Supported by NSF grant \#1701567}
\date{September 2020}

\begin{document}

\begin{abstract}
	Given a pair of elliptic curves $E\one,E\two$ over a field $\numfield$, we have a natural map $\CH^1(E\one)_0\otimes\CH^1(E\two)_0\to\CH^2(E\one\times E\two)$, and a conjecture due to Bloch and Beilinson predicts that the image of this map is finite when $\numfield$ is a number field. We construct a $2$-parameter family of elliptic curves that can be used to produce examples of pairs $E\one,E\two$ where this image is finite. The family is constructed to guarantee the existence of a rational curve passing through a specified point in the Kummer surface of $E\one\times E\two$.
\end{abstract}

\maketitle

\section{Introduction}

Given a smooth irreducible projective variety $X$ over a field $\numfield$, let $\CH^r(X)$ denote the Chow group of cycles of codimension $r$ on $X$ modulo rational equivalence (see for example \cite{fulton}). If $X$ has dimension $d$, let $\CH^d(X)_0$ denote the subgroup of $\CH^d(X)$ consisting of zero-cycles of degree $0$.

If $E\one$ and $E\two$ are elliptic curves over $\numfield$, we have an Abel-Jacobi map
\begin{align*}
\text{AJ}:\CH^2(E\one\times E\two) &\to (E\one\times E\two)(\numfield)\\
\sum a_i[R_i]&\mapsto \sum a_iR_i.
\end{align*} 
(Some care is needed if the points $R_i$ are not defined over $\numfield$; see Section~\ref{context}.) A celebrated conjecture independently due to Bloch~\cite{blochconj} and Beilinson~\cite{beilinsonconj} predicts that $\ker\text{AJ}$ is finite when $\numfield$ is a number field. To this date, there is very little concrete evidence for this conjecture. See Section~\ref{context} for further discussion.

We will prove an implication of this conjecture for a family of curves. Consider the following map:
\begin{align*}
\Phi:\CH^1(E\one)\otimes \CH^1(E\two) &\to  \CH^2(E\one\times E\two)\\
[P\one]\otimes [P\two] &\mapsto [(P\one,P\two)].
\end{align*}
In terms of the projection maps $\pi_i:E\one\times E\two\to E_i$ and the intersection product on $E\one\times E\two$, we can equivalently write $\Phi(D\one\otimes D\two)=\pi\one^*(D\one)\cdot \pi\two^*(D\two)$. Within the domain of this map is the subgroup $\CH^1(E\one)_0\otimes \CH^1(E\two)_0$, which is isomorphic to $E\one(\numfield)\otimes E\two(\numfield)$ and is therefore infinite when $E\one$ and $E\two$ both have positive rank. We can check (see Section~\ref{context}) that $\Phi$ maps this subgroup into the kernel of $\text{AJ}$, and so Beilinson's conjecture predicts that the image of this subgroup should be finite. We summarize this situation with the following definition.
\begin{defn}
	We say that the product $E\one\times E\two$ is \emph{clean} if the image under $\Phi$ of $\CH^1(E\one)_0\otimes \CH^1(E\two)_0$ is finite. In this case we also say that $E\one$ and $E\two$ form a \emph{clean pair}.
\end{defn}
\noindent
In this language, Bloch and Beilinson's conjecture implies that all pairs of elliptic curves over a number field are clean.

We will construct a family of elliptic curves which can be used to produce nontrivial (i.e. positive rank) clean pairs. Let $\scE$ be the elliptic curve over $\numfield(\S ,\T)$ given by
\[\scE:y^2=x^3-3\T ^2x+2\T ^3+\left(1-\S-3\T\right)^2 \S,\]
and let $E_{\s ,\t}$ denote the specialization of $\scE$ obtained by substituting $\s,\t\in\numfield$ for the indeterminates $\S$ and $\T$.

\begin{thm}\label{cleanthm}
	Let $\numfield$ be an infinite field with $\chr\numfield\neq 2,3$, and assume that $|E(\numfield)_{\text{tors}}|$ is uniformly bounded for all elliptic curves $E$ over $\numfield$. There is a nonempty Zariski-open subset $\zaropen$ of $\bbA_\numfield^2$ such that for all $(\s,\t\one),\,(\s,\t\two)\in\zaropen(\numfield)$, if $E_{\s,\t\one}$ and $E_{\s,\t \two}$ are rank $1$ elliptic curves, then $E_{\s ,\t\one}\times E_{\s ,\t \two}$ is clean.
\end{thm}

\begin{rmk}
	The curve $\scE$ has a specified rational point
	\[\scP:=(1-S-2T,\,1-S-3T)\in \scE(\numfield(\S ,\T)).\]
	The Zariski-open subset $\zaropen$ of Theorem~\ref{cleanthm} is the locus where $\scE$ has good reduction and the reduction of $\scP$ has infinite order. So for all $(\s,\t)\in\zaropen(\numfield)$, $E_{\s,\t}$ is a positive rank elliptic curve.
\end{rmk}

The proof of Theorem~\ref{cleanthm} will be given in Sections~\ref{kummerrational} and~\ref{parametrization}. When $\numfield$ is a number field, Merel proved that the torsion subgroup of $E(\numfield)$ is uniformly bounded~\cite{merel}, so we obtain the following corollary.
\begin{cor}
	Let $\numfield$ be a number field. There is a nonempty Zariski-open subset $\zaropen$ of $\bbA_\numfield^2$ such that for all $(\s,\t\one),\,(\s,\t\two)\in\zaropen(\numfield)$, if $E_{\s,\t\one}$ and $E_{\s,\t \two}$ are rank $1$ elliptic curves, then $E_{\s ,\t\one}\times E_{\s ,\t \two}$ is clean.
\end{cor}

For any nonzero $\s\in\numfield$, let $\scE_\s$ be the restriction of $\scE$ to $\S=\s$; that is, $\scE_\s$ is the curve over $\numfield(\T)$ defined by
\[\scE_\s:y^2=x^3-3\T ^2x+2\T ^3+\left(1-\s-3\T\right)^2 \s.\]
By Theorem~\ref{cleanthm}, any two rank $1$ specializations of $\scE_\s$ subject to a certain Zariski-open condition will form a clean pair. We would like to understand how many elliptic curves that satisfy these conditions. 

\begin{cor}\label{infwithE}
	Let $\numfield$ be an infinite field with $\chr\numfield\neq 2,3$, and assume that $\scE_\s$ does not have elevated rank for any nonzero $\s\in\numfield$. Let $E$ be any rank $1$ elliptic curve over $\numfield$ of the form $y^2=x^3-3\t ^2x+b$, such that there is no torsion point in $E(\numfield)$ with $x$-coordinate equal to $\t$. Then there is an infinite collection of elliptic curves $E'$ over $\numfield$, no two of which are isomorphic over $\overline{\numfield}$, such that $E\times E'$ is clean.
\end{cor}
Corollary~\ref{infwithE} will be proven in Section~\ref{infpairs}. The conditions on $E$ are to guarantee that $E$ is isomorphic to a specialization of $\scE_\s$ for an appropriate choice of $\s$.
\begin{rmk}
	For a definition of elevated rank, see Definition~\ref{defn:elevatedrank}. We discuss the elevated rank hypothesis of Corollary~\ref{infwithE} in Section~\ref{infpairs}. In particular, the assumption that $\scE_\s$ does not have elevated rank is used to conclude that $E_{s,t}$ has the same rank as $\scE_\s$ for infinitely many $t\in\numfield$. We are not able to prove this hypothesis for any $\scE_\s$, but it seems likely to always hold when $\numfield$ is a number field (Remark~\ref{rmk:noelevatedrank}).
\end{rmk}

Despite not being able to prove unconditionally that these collections are infinite, we can easily use these families to generate many clean pairs of curves, as will be discussed in Section~\ref{examples}. In particular, we compute a list of rank $1$ curves over $\zaropen(\QQ)$ with $\S=1$, from which we obtain approximately $7\cdot 10^8$ nontrivial clean pairs of rank $1$ curves. 

\subsection{Acknowledgments}

The author would like to thank Akshay Venkatesh for drawing his attention to this problem, for providing many potential strategies to try, and for pointing him to the prior work of Kartik Prasanna and Vasudevan Srinivas, which inspired his work on this problem. The specific strategy of looking for rational curves in the Kummer surface developed out of conversations with Ravi Vakil. Many thanks also to the anonymous reviewers for catching errors and providing suggestions to improve the exposition.  

\section{Context}\label{context}

Let $X$ be a surface over $\numfield$. The Chow group $\CH^2(X)$ depends quite strongly on the field $\numfield$; in general, $\CH^2(X)$ can be extremely unwieldy. This was first shown by Mumford, who proved that if $X$ is defined over $\numfield=\CC$ and has a nonzero holomorphic $2$-form (this includes for example $X=E\one\times E\two$), then $\CH^2(X)$ is ``infinite-dimensional;'' that is, for any positive integer $n$, if a subvariety of $\Sym^n(X)$ is sent to a single point under the map
\begin{align*}
\Sym^n(X)&\to \CH^2(X)\\
\{P_i\}&\mapsto \sum_i [P_i],
\end{align*}
then this subvariety must have codimension at least $n$~\cite[Corollary]{mumford}. 

When the field of definition is a number field, it is believed $\CH^2(X)$ is much more well-behaved. Letting $\text{Alb}(X)$ denote the Albanese variety of $X$ and fixing a base point $P_0\in X(\numfield)$,\footnote{For simplicity, we only consider the case that $X(\numfield)$ is nonempty.} we have a natural surjection $\text{AJ}:\CH^2(X)\to \text{Alb}(X)(\numfield)$, defined as follows. Given a closed point $R$ of $X$, its residue field $\numfield(R)$ is a finite extension of $\numfield$. If $\numfield(R)=\numfield$, simply define $\text{AJ}([R])=R-P_0$. Otherwise, $R$ splits over its residue field: $R\otimes_\numfield \numfield(R)=\{R^{(1)},\ldots,R^{(m)}\}$, where $R^{(1)},\ldots,R^{(m)}\in X(\numfield(R))$ are conjugate under $\Gal(\numfield(R)/\numfield)$. The sum $R^{(1)}+\cdots +R^{(m)}-mP_0\in\text{Alb}(X)(\numfield(R))$ is Galois-invariant, and hence descends to a point $\text{AJ}(R)\in \text{Alb}(X)(\numfield)$. The map $\text{AJ}$ can then be extended to $\CH^2(X)$ by linearity.

A far-reaching set of conjectures due to Beilinson~\cite{beilinsonconj} and Bloch~\cite{blochconj} imply that when $\numfield$ is a number field, $\CH^2(X)$ is finitely generated, with rank equal to the rank of $\text{Alb}(X)(\numfield)$ (this implication is described in the case $\numfield=\QQ$ in~\cite[Lemma 5.1]{Beilinson1987}). That is, $\text{AJ}$ is conjecturally an isomorphism modulo a finite kernel. These conjectures were made with little concrete evidence (as Beilinson notes immediately after~\cite[Conjecture 5.0]{Beilinson1987}), and since then there are still very few cases for which the conjecture is known to be true. Bloch gives examples of rational surfaces satisfying the conjecture~\cite[Chapter 7]{blochlect}, but the author is not aware of any non-rational surfaces that are known to satisfy the conjecture.

Now we return to the special case $X=E\one\times E\two$; note that $\text{Alb}(X)=X$ because $X$ is an abelian variety. The Chow group $\CH^2(E\one\times E\two)$ is generated by closed points of $E\one\times E\two$; these points may not be defined over $\numfield$, and one major difficulty in studying the Chow group comes from relations involving high-degree points. However, even the question of which cycles supported at $\numfield$-points are rationally equivalent to zero is not fully understood, and this is the question this paper addresses. All cycles of this form lie in the image of the map $\Phi$ defined above: given any $\numfield$-point $(P\one,P\two)\in (E\one\times E\two)(\numfield)$, the corresponding zero-cycle is $\Phi([P\one]\otimes [P\two])$. 

Within the domain of $\Phi$ is the subgroup $\CH^1(E\one)_0\otimes \CH^1(E\two)_0$ generated by elements of the form $([P\one]-[Q\one])\otimes ([P\two]-[Q\two])$ for $P\one,Q\one\in E\one(\numfield)$ and $P\two,Q\two\in E\two(\numfield)$, and we have
\begin{align*}
&(\text{AJ}\circ \Phi)(([P\one]-[Q\one])\otimes ([P\two]-[Q\two]))\\
&\qquad=(P\one,P\two)-(P\one,Q\two)-(Q\one,P\two)+(Q\one,Q\two)\\
&\qquad=0.
\end{align*}
Hence $\Phi(\CH^1(E\one)_0\otimes \CH^1(E\two)_0)$ is a subgroup of $\ker\text{AJ}$ and is therefore conjecturally finite; if this holds, we say that $E\one\times E\two$ is \emph{clean}. Intuitively, this says that given any relation among points in $(E\one\times E\two)(\numfield)$, some nonzero multiple of this relation can be expressed as a rational equivalence.

Prior to this work, Prasanna and Srinivas developed a technique using Heegner points on a modular curve to prove that certain pairs of rank $1$ curves are clean~\cite{prasanasrinivas}. Their technique requires $E\one$ and $E\two$ to have the same conductor, and must be applied on a case-by-case basis (their preprint proves cleanliness of two pairs of curves). Our contribution is to provide a two-parameter family of curves for which there is a simple test for clean pairs: if two curves are contained in a certain Zariski-open, have rank $1$, and share a common value for the first parameter, then the pair of curves is clean. Note that this makes modest progress towards providing evidence for Bloch and Beilinson's conjectures, but is still far from providing any example demonstrating the truth of the conjectures, because (as discussed above) $\ker\text{AJ}$ is generated by points of arbitrarily large degree over $\numfield$.

\section{A Pencil of Cubic Curves in the Kummer Surface}\label{kummerrational}

We henceforth assume $\text{char}\numfield\neq 2,3$, so every elliptic curve over $\numfield$ has a short Weierstrass form.

Let $E\one$ and $E\two$ be elliptic curves over $\numfield$, with respective identity points $O\one$ and $O\two$. The product $E\one\times E\two$ has an involution $\iota$ given by negation, which acts freely away from the $2$-torsion points of $E\one\times E\two$. We can form the quotient by $\iota$, called the \emph{Kummer surface} $K$ of $E\one\times E\two$ (see, for example, Section 10.3 of~\cite{dolgachev}), and we have a degree $2$ morphism $\pi:E\one\times E\two\to K$ with $\pi=\pi\circ\iota$. The Kummer surface has sixteen singularities, corresponding to the fixed points of $\iota$; blowing up these sixteen points gives a smooth surface $\widehat{K}$. Since $K$ and $\widehat{K}$ are birationally equivalent, $\pi$ induces a rational map $\widehat{\pi}:E\one\times E\two\dashedrightarrow \widehat{K}$, defined away from the fixed points of $\iota$.

Let $E\one$ and $E\two$ have Weierstrass forms $y\one^2=f(x\one)$ and $y\two^2=g(x\two)$ respectively. The hypersurface in $\bbA^3(x\one,x\two,\r)$ defined by
\[f(x\one)=\r^2g(x\two)\]
is an affine model for $\widehat{K}$, with the rational map $\widehat{\pi}$ given in these coordinates by $(x\one,\,y\one,\,x\two,\,y\two)\mapsto (x\one,\,x\two,\,y\one/y\two)$. The map
\begin{align*}
\widehat{K}&\to\PP^1\\
(x\one,\,x\two,\,\r)&\mapsto \r
\end{align*}
gives $\widehat{K}$ the structure of an elliptic surface; the fiber over a point $\r\in\PP^1(\numfield)$ is a cubic curve $C_\r$. This fibration is known as \emph{Inose's pencil}~\cite{shioda07}. In general, the fiber $C_\r$ will be a genus $1$ curve, but if $C_\r$ has a singularity then it will be a rational curve (or a union of rational curves). 

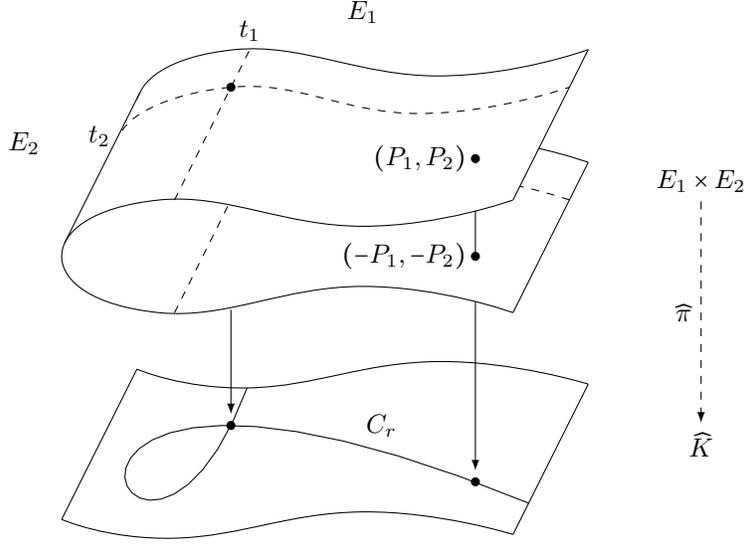
\begin{figure}
	\begin{center}
		\begin{tikzpicture}[domain=-1.665:1.13]
		
		\node at (4.5,3.25) {$E\one$};
		\node at (0,1.5) {$E\two$};
		\node (EE) at (9,1) {$E\one\times E\two$};
		\node (K) at (9,-2.5) {$\widehat{K}$};
		\draw[dashed, -latex] (EE) to node[anchor=east] {$\widehat{\pi}$} (K);

		\draw  plot[smooth, tension=.8,xshift=1cm,yshift=2cm] coordinates {(6.5,-.75) (4.5,-.4) (1.8,-.75) (0.5,0) (1.8,0.75) (4.5,.4) (6.5,.75)};
		\path[fill=white] (1.54,2.16) -- (6.5, 2) -- (7.5, 2.75) -- (6.5,.75) -- (0.5,0) -- cycle;

		\draw[dashed]  plot[smooth, tension=.8,xshift=0.75cm,yshift=1.5cm] coordinates {(6.5,-.75) (4.5,-.4) (2,-.75) (0.5,0) (2,0.75) (4.5,.4) (6.5,.75)};
		\path[fill=white] (1.29,1.66) -- (6.5, 1.66) -- (7.25, 2.25) -- (6.5,.75) -- (0.5,0) -- cycle;
		\node at (1,1.6) {$\t\two$};

		\draw (0.54,0.18) -- (1.54, 2.18);
		\draw (6.5,.75) -- (7.5, 2.75);
		\draw (6.5,-.75) -- (7.5, 1.25);

		\draw[dashed] (2,-.75) -- (3,1.25);
		\path[fill=white] (2,1) -- (3,1.25) -- (3.5,0.4) -- cycle;
		\draw[dashed] (2,.75) -- (3,2.75);
		\node at (3,3) {$\t\one$};

		\draw[-latex] (6,1.3) -- (6,-2.85);
		\path[fill=white] (6,1.3) -- (6.1,0.6) -- (5.9,0.61) -- cycle;
		\path[fill=white] (6,0) -- (6.1,-.6) -- (5.9,-0.61) -- cycle; 
		\filldraw[black] (6,1.3) circle (1.5pt) node[anchor=east] {$(P\one,P\two)$};
		\filldraw[black] (6,0) circle (1.5pt) node[anchor=east] {$(-P\one,-P\two)$};
		\filldraw[black] (6,-3) circle (1.5pt);

		\filldraw[black] (2.75,2.25) circle (1.5pt);
		\draw[-latex] (2.75,-.71) -- (2.75,-2.1);
		\filldraw[black] (2.75,-2.25) circle (1.5pt);

		\draw  plot[smooth, tension=.8] coordinates {(6.5,-.75) (4.5,-.4) (1.8,-.75) (0.5,0) (1.8,0.75) (4.5,.4) (6.5,.75)};

		\draw (1.5,-1.5) -- (0.5,-3.5);
		\draw (7.5,-1.7) -- (6.5,-3.7);
		\draw  plot[smooth, tension=.8, xshift=1cm, yshift=-1.5cm] coordinates {(6.5,-.2) (4.5,.1) (2,-.25) (0.5,0)};
		\draw  plot[smooth, tension=.8, yshift=-3.5cm] coordinates {(6.5,-.2) (4.5,.1) (2,-.25) (0.5,0)};

		\draw plot[xshift=2.75cm, yshift=-2.25cm, rotate=31.5] ({1.6*(\x * \x - 1)}, {1*(\x^3 - \x)});
		\node at (4.75,-2.25) {$C_\r$};
		
		\end{tikzpicture}
	\end{center}
	
	\caption{The curves $C_\r:f(x\one)=\r^2g(x\two)$ form a fibration of $\widehat{K}$. If $f'(\t\one)=0$ and $g'(\t\two)=0$, a curve $C_\r$ passing through a point of the form $(\t\one,\t\two,\r)$ will be singular.}
	\label{mainideadiagram}
\end{figure}

Now suppose $E\one$ and $E\two$ are rank $1$ curves. If the image of a point $(P\one,P\two)\in (E\one\times E\two)(\numfield)$ lies on $C_\r$, then we will be able to generate rational equivalences on $E\one\times E\two$ involving $(P\one,P\two)$ by pulling back principal divisors on $C_\r$. If we can guarantee $C_\r$ is singular (see Figure~\ref{mainideadiagram}), then every degree $0$ divisor not containing the singularity will be principal, giving us the flexibility we need to prove that $E\one\times E\two$ is clean. 

This idea is made precise in Proposition~\ref{PQuusamefiber}. We first state and prove a lemma that we will use to construct divisors.

\begin{lem}\label{lem:ptatinfty}
	For any $\r\in\numfield\setminus\{0\}$, let $\overline{C_\r}\subseteq K$ denote the image of $C_\r$ under the normalization map $\widehat{K}\to K$, and let $A\subseteq \overline{C_\r}$ denote the image of the affine subset $C_\r\cap \bbA^3$. Then $\overline{C_\r}\setminus A$ consists of the single point $\pi(O\one, O\two)$.
\end{lem}
\begin{proof}
	The closed morphism $E\one\times E\two\to\PP^1\times\PP^1$ given by $(x\one,\,y\one,\,x\two,\,y\two)\mapsto (x\one,\,x\two)$ factors through the surjective map $\pi$, and so we have a closed morphism $\psi:K\to \PP^1\times\PP^1$. The curve $C_\r$ is closed because it is a fiber over a closed point of $\PP^1$, and the normalization map $\widehat{K}\to K$ is closed, so $\psi(\overline{C_\r})$ is closed in $\PP^1\times\PP^1$.
		
	On the other hand, the image of $A$ under $\psi$ is the affine curve defined by $f(x\one)=\r^2g(x\two)$. This is not closed in $\PP^1\times\PP^1$, so $\overline{C_\r}\setminus A$ is nonempty. Since $\r\neq 0$, the closure of $\psi(A)$ in $\PP^1\times\PP^1$ contains no points at infinity other than $(\infty, \infty)$. Hence $\overline{C_\r}\setminus A$ is contained in $\psi^{-1}(\{(\infty, \infty)\})=\{\pi(O\one, O\two)\}$.
\end{proof}

\begin{prop}\label{PQuusamefiber}
	Let $E\one$ and $E\two$ be rank $1$ elliptic curves with Weierstrass equations $y\one^2=f(x\one)$ and $y\two^2=g(x\two)$ respectively, and let $\t\one,\t\two\in\overline{\numfield}$ satisfy $f'(\t\one)=g'(\t\two)=0$. Suppose there exist non-torsion points $P\one\in E\one(\numfield)$ and $P\two\in E\two(\numfield)$ with $f(\t\one)/y\one(P\one)^2=g(\t\two)/y\two(P\two)^2$. Then $E\one\times E\two$ is clean (i.e. $\Phi(\CH^1(E\one)_0\otimes \CH^1(E\two)_0)$ is finite).
\end{prop}

\begin{proof}
	Let $\r=y\one(P\one)/y\two(P\two)$. Since we have $\r^2=y\one(P\one)^2/y\two(P\two)^2=f(\t\one)/g(\t\two)$, the curve $C_\r:f(x\one)=\r^2g(x\two)$ in $\widehat{K}$ contains both $\widehat{\pi}(P\one,P\two)$ and $(\t\one,\,\t\two,\,\r)$. Since $f'(\t\one)=g'(\t\two)=0$, the point $(\t\one,\,\t\two,\,\r)$ is a singularity of $C_\r$. Hence $\overline{C_\r}$, the image of $C_\r$ in $K$, is a rational curve containing $\pi(P\one,P\two)$ and $\pi(O\one,O\two)$. The cycle $[\pi(P\one,P\two)]-[\pi(O\one,O\two)]$ is therefore a principal divisor $(h)$ for some $h$ in the function field of $\overline{C_\r}$; explicitly, in the setup of Lemma~\ref{lem:ptatinfty}, we can take a rational function on the singular affine cubic $\psi(A)$ that has no zeroes or poles aside from a simple zero at $(x\one(P\one),x\two(P\two))$, and pull this back along $\psi$ to obtain $h$. Pulling back $h$ along $\pi$, we obtain a rational function on $\pi^{-1}(\overline{C_\r})$. 
	
	In the same way, $C_{-\r}$ contains both $\widehat{\pi}(-P\one,P\two)$ and the singularity $(\t\one,\,\t\two,\,-\r)$. We conclude that the zero-cycles
	\begin{align*}&[(P\one,P\two)]+[(-P\one,-P\two)]-2[(O\one,O\two)],\\
	&[(-P\one,P\two)]+[(P\one,-P\two)]-2[(O\one,O\two)]
	\end{align*}
	are principal divisors on the curves $\pi^{-1}(\overline{C_\r})$ and $\pi^{-1}(\overline{C_{-\r}})$, respectively. Their difference,
	\begin{align*}
		&[(P\one,P\two)]-[(-P\one,P\two)]-[(P\one,-P\two)]+[(-P\one,-P\two)]\\
		&\qquad=\Phi\left(([P\one]-[-P\one])\otimes ([P\two]-[-P\two])\right),
	\end{align*}
	is therefore zero in $\CH^2(E\one\times E\two)$.
	
	Now take any $D\one\in\CH^1(E\one)_0$ and $D\two\in \CH^1(E\two)_0$. Since $E\one$ has rank $1$, there will exist integers $n\neq 0$ and $m$ such that $nD\one=m([P\one]-[O\one])$; using the rational equivalence $[P\one]+[-P\one]-2[O\one]=0$ in $\CH^1(E\one)$, we have $2nD\one=m([P\one]-[-P\one])$. Likewise, $2n' D\two$ will be a multiple of $[P\two]-[-P\two]$ for some nonzero integer $n'$, and so $4nn'\Phi(D\one\otimes D\two)$ is zero in $\CH^2(E\one\times E\two)$. Since $\CH^1(E\one)_0\otimes\CH^1(E\two)_0$ is finitely generated by elements of the form $D\one\otimes D\two$, this proves $E\one\times E\two$ is clean.
\end{proof}

Although Proposition~\ref{PQuusamefiber} applies as long as $\t\one,\t\two\in\overline{\numfield}$, from the next section onward we will assume $\t\one,\t\two\in\numfield$. The following result deals with the remaining cases.

\begin{lem}
	Assume the conditions of Proposition~\ref{PQuusamefiber}, and further assume that either $\t\one$ or $\t\two$ is not in $\numfield$. Then $E\one$ and $E\two$ are isomorphic over $\numfield$.
\end{lem}

\begin{proof}
	Letting $f(x\one)=x\one^3+a\one x\one+b\one$ and $g(x\two)=x\two^3+a\two x\two+b\two$, the conditions $f'(\t\one)=g'(\t\two)=0$ imply $a\one=-3\t\one^2$ and $a\two=-3\t\two^2$. Letting $\r=y\one(P\one)/y\two(P\two)\in\numfield$, the condition $\r^2=f(\t\one)/g(\t\two)$ then implies
	\[\frac{2a\one\t\one}{3}+b\one=\r^2\left(\frac{2a\two\t\two}{3}+b\two\right).\]
	If $\t\one\notin\numfield$ or $\t\two\notin\numfield$, then we can conclude that $\numfield(\t\one)=\numfield(\t\two)$ is a quadratic extension of $\numfield$, with a Galois automorphism acting by $\t\one\mapsto -\t\one$ and $\t\two\mapsto -\t\two$. This means we must have $a\one\t\one=\r^2a\two\t\two$ (which, when squared, implies $a\one^3=\r^4a\two^3$) and $b\one=\r^2b\two$. If we set $d=\frac{a\one}{a\two \r}$, these equations imply $a\one=d^4a\two$ and $b\one=d^6b\two$; that is, $E\one$ and $E\two$ are isomorphic over $\numfield$.
\end{proof}

\section{Parametrization}\label{parametrization}

We seek to parameterize pairs of elliptic curves satisfying the conditions of Proposition~\ref{PQuusamefiber}. By the following lemma, we can assume that these curves have a particular form.

\begin{lem}
	Let $E$ be an elliptic curve with Weierstrass equation $y^2=f(x)$, $\t\in\numfield$ such that $f'(\t)=0$, and $P\in E(\numfield)$ non-torsion. Set $\s:=f(\t)/y(P)^2$. Then $E$ is isomorphic to the curve $E_{\s,\t}$ with Weierstrass equation
	\[y^2=f_{\s,\t}(x):=x^3-3\t ^2x+2\t ^3+\left(1-\s-3\t\right)^2 \s.\]
	Further, we have $f_{\s,\t}'(\t)=0$, and the point $P_{\s,\t}:=(1-\s-2\t ,\, 1-\s -3\t )\in E_{\s,\t}(\numfield)$ is a non-torsion point satisfying $\s=f_{\s,\t}(\t)/y(P_{\s,\t})^2$.	
\end{lem}
\begin{proof}
	Note that $f$ is uniquely determined by $\t$ and $P$:
	\[f(x)=x^3-3\t^2x+(y(P)^2-x(P)^3+3\t^2 x(P)).\]
	If $x(P)=\t$, then $-2P=(-2x(P),y(P))$ is again a non-torsion point, and $\t$ and $-2P$ determine the same curve and the same value of $\s$ as $t$ and $P$. So replacing $P$ with $-2P$ if necessary, we may assume without loss of generality that $x(P)\neq \t$. 
	
	For any nonzero $d\in\numfield$, the substitution $(x,y)\mapsto (d^2x,d^3y)$ results in an isomorphic curve, determined by $d^2\t$ and $(d^2x(P),d^3y(P))$, and $\s$ is preserved. By setting $d=\frac{x(P)-\t}{y(P)}$, we can assume without loss of generality that $y(P)=x(P)-\t$. 

	Now from $f(\t)=\s y(P)^2$ we obtain
	\begin{align*}
		y(P)^2-x(P)^3+3\t^2 x(P) = 2\t^3+\s y(P)^2,
	\end{align*}
	or rearranging,
	\[(1-\s)y(P)^2=(x(P)-\t )^2(x(P)+2\t ).\]
	By the assumption $y(P)=x(P)-\t$, this simplifies to $x(P)=1-\s-2\t$.
\end{proof}

This leads us to the elliptic curve $\scE$ over $\numfield(\S,\T)$ defined in the introduction, which has a distinguished point $\scP\in\scE(\numfield(\S,\T))$:
\begin{align*}
\scE :y^2&=(x-\T )^2(x+2\T )+\left(1-\S -3\T\right)^2\S ,\\
\scP &:=\left(1-\S-2\T ,\, 1-\S -3\T \right).
\end{align*}
To complete the proof of Theorem~\ref{cleanthm}, it suffices to determine the pairs of specializations of $\scE$ that satisfy the conditions of Proposition~\ref{PQuusamefiber}.

\begin{proof}[Proof of Theorem~\ref{cleanthm}]
	The equation for $\scE$ also defines a hypersurface $\widetilde{\scE}$ in $\PP_\numfield^2(x,y)\times\bbA_\numfield^2(S,T)$. Let $\zaropen_0$ be the Zariski-open subset of $\bbA_\numfield^2$ on which the discriminant
	\[\Delta(\scE)=-432 S (1 - S - 3 T)^2 (4 T^3 + (1 - S - 3 T)^2 S)\]
	is nonzero; this is nonempty as long as $\chr\numfield\neq 2,3$. Then the fiber of the projection $\widetilde{\scE}\to\bbA_\numfield^2$ over a point $(\s,\t)\in\zaropen_0(\numfield)$ will be an elliptic curve $E_{\s,\t}$ over $\numfield$.
	
	Each element of $\scE(\numfield(\S,\T))$ determines a section $\bbA_\numfield^2\to \widetilde{\scE}$. Let $\widetilde{\scO}$ denote the image of the zero section $\bbA_\numfield^2\to\widetilde{\scE}$, and for each integer $\ell\geq 1$, let $\widetilde{\ell\scP}$ denote the image of the section associated to $\ell\scP$. Pulling back $\widetilde{\scO}\cap\widetilde{\ell\scP}$ along the zero section, we obtain a closed subvariety $\scZ_\ell$ of $\bbA_\numfield^2$, where a point $(\s,\t)\in\zaropen_0(\numfield)$ is in $\scZ_\ell(\numfield)$ if and only if $\ell P_{\s ,\t}$ is the identity of $E_{\s,\t}$. The point $\scP$ is not itself torsion (one way to see this is to specialize to $\S=1$ and show that the canonical height is nonzero; this computation is carried out in Appendix~\ref{proofgenrank1}), so $\scZ_\ell$ is not all of $\bbA_\numfield^2$. Its complement, which we denote $\zaropen_\ell$, is therefore a non-empty Zariski-open subset.
	
	By our hypothesis of uniform boundedness for torsion, there exists an integer $L$ such that if $P_{\s ,\t}$ is torsion in $E_{\s ,\t}$ for any $(\s,\t)\in\zaropen_0(\numfield)$, it must have order $1\leq\ell\leq L$. Hence, the finite intersection 
	\[\zaropen := \bigcap_{\ell=0}^L\zaropen_\ell\]
	is a non-empty Zariski-open set such that $P_{\s ,\t}$ is non-torsion for all $(\s,\t)\in \zaropen(\numfield)$.
	
	Suppose we take any $(\s,\t\one),(\s,\t\two)\in\zaropen(\numfield)$ such that $E_{\s,\t\one}$ and $E_{\s,\t\two}$ are rank $1$ curves. By definition of $\zaropen$, the points $P\one=P_{\s,\t\one}$ and $P\two=P_{\s,\t\two}$ will not be torsion, and we will have $f(\t\one)/y\one(P\one)^2=\s=g(\t\two)/y\two(P\two)^2$. We also have $f'(\t\one)=g'(\t\two)=0$ directly from the definition. Hence, by Proposition~\ref{PQuusamefiber}, $E_{\s,\t\one}\times E_{\s,\t\two}$ is clean.
\end{proof}

\section{Infinitely Many Clean Pairs?}\label{infpairs}

For each $\s\in\numfield\setminus\{0\}$, let $\scE_\s$ be the curve over $\numfield(\T)$ obtained from $\scE$ by evaluating the indeterminate $\S$ at $\s$. In this section we prove Corollary~\ref{infwithE}: assuming that none of the curves $\scE_\s$ have elevated rank, then for any rank $1$ elliptic curve $E$ satisfying certain conditions, it will form a clean pair with infinitely many $E'$ that are non-isomorphic over $\overline{\numfield}$.  

We begin by defining what it means for a curve over $\numfield(\T)$ to have elevated rank. After a brief discussion of this phenomenon, we will proceed with a proof of Corollary~\ref{infwithE}.

\begin{defn}\label{defn:elevatedrank}
	Let $\scF$ be an elliptic curve over $\numfield(\T)$, and let $F_\t$ denote the specialization of $\scF$ at $\T=\t$. We say that $\scF$ \emph{has elevated rank} if for all but finitely many $\t\in\numfield$, the rank of $F_\t(\numfield)$ is strictly greater than the rank of $\scF(\numfield(\T))$. 
\end{defn}

\begin{rmk}\label{rmk:noelevatedrank}
	Let us consider the phenomenon of elevated rank over various fields $\numfield$. Conrad, Conrad, and Helfgott~\cite{conradconrad} describe examples of curves with elevated rank over $\QQ(\T)$, but point out that all known examples are isotrivial (the $j$-invariant is constant). In fact, assuming the parity, density, squarefree-value, and Chowla conjectures, they prove that every curve over $\QQ(\T)$ with elevated rank must be isotrivial. In contrast, they construct examples of nonisotrivial curves of elevated rank over $\numfield(\T)$, for $\numfield$ a field of positive characteristic. These examples depend very strongly on the characteristic being nonzero; as the authors mention, ``the failure of Chowlaâ€™s conjecture in positive characteristic was our initial clue to the possibility that elevated rank may occur in nonisotrivial families in the function field case''~\cite[36]{conradconrad}. Hence we suspect that there should be no nonisotrivial curves of elevated rank over $\numfield(\T)$ when $\numfield$ is a number field. Since $\scE_\s$ is nonisotrivial for all $\s\neq 0$, the parity, density, squarefree-value, and Chowla conjectures imply that $\scE_\s$ never has elevated rank when $\numfield=\QQ$, and it is plausible that $\scE_\s$ never has elevated rank over any number field.
\end{rmk}

\begin{proof}[Proof of Corollary~\ref{infwithE}]
	Let $E$ be a rank $1$ elliptic curve of the form $y^2=x^3-3\t ^2x+b$, such that there is no torsion point in $E(\numfield)$ with $x$-coordinate equal to $t$. If $b-2\t^3=r^2$ for some $r\in\numfield$, set $P=(-2\t,r)$ (which will be non-torsion by assumption); if $b-2\t^3$ is not in $\numfield^2$, let $P\in E(\numfield)$ be any non-torsion point. By the techniques of Section~\ref{parametrization}, if we set $\s =\frac{b-2\t ^3}{y(P)^2}$ and $d=\frac{x(P)-\t}{y(P)}$, then we will have $E\cong E_{\s ,d^2\t}$ (with $P$ corresponding to $P_{\s,d^2\t}$), and $(\s,d^2\t)\in\zaropen(\numfield)$.
	
	Now consider the curve $\scE_\s$. By our definition of $\s$ and choice of $P$, we either have $\s=1$ (if $b-2\t^3=r^2$) or $\s$ is not in $\numfield^2$. We will apply the following rank computation, which can be found immediately after the current proof.
	\begin{prop}\label{genrank1}
		The group $\scE_\s(\numfield(\T))$ has rank $1$ for all $\s \in\numfield\setminus\numfield^2$ and for $\s =1$, and has rank $2$ for $\s \in\numfield^2\setminus\{0,1\}$. 
	\end{prop}
	\noindent	
	By this result, $\scE_\s(\numfield(\T))$ has rank $1$ for all $\s$ we are considering. A theorem of Silverman tells us that there are only finitely specializations of $\scE_\s$ that have rank lower than the generic rank~\cite[Theorem C]{silverman_heights}, and we are assuming that $\scE_\s$ does not have elevated rank, so there are infinitely many $\t'\in\numfield$ such that $E_{\s,\t'}$ has rank equal to $1$.
	
	The restriction of the Zariski-open $\zaropen$ to the line $\S=\s$ in $\mathbb{A}_\numfield^2$ is nonempty (it contains $E_{\s ,d^2\t}$), so $(\s,\t')\in\zaropen(\numfield)$ for all but finitely many $\t'\in\numfield$. Hence, by Theorem~\ref{cleanthm}, there are infinitely many $E_{\s,\t'}$ with rank $1$ that will form a clean pair with $E$. The $j$-invariant of $\scE_\s$,
	\[j(\scE_\s)=\frac{6912T^6}{s(1-s-3T)^2(s(1-s-3T)^2+4T^3)},\]
	is a nonconstant rational function in $T$, so any given $j$-invariant in $\numfield$ is attained by a specialization of $\scE_\s$ only finitely many times. Thus one can find infinitely many $E_{\s,\t'}$ as above with pairwise distinct $j$-invariants.
\end{proof}

\begin{proof}[Proof of Proposition~\ref{genrank1}]
	Let $\overline{\scE_\s}$ denote the base change of $\scE_\s$ to $\overline{\numfield}(\T)$, and let $\widetilde{\scE_\s}$ denote the minimal elliptic surface over $\PP^1_{\overline{\numfield}}$ associated to $\overline{\scE_\s}$. Since $\widetilde{\scE_\s}$ is a rational elliptic surface (for instance by~\cite[Remark 1.3.1]{rosensilverman}) over an algebraically closed field, a special case of the Shioda-Tate Theorem \cite[Theorem 10.3]{shioda90} tells us that the rank of $\overline{\scE_\s}(\overline{\numfield}(\T))$ will equal $8-\sum_{\t \in\badred} (m_\t -1)$, where $\badred$ is the set of places of bad reduction, and $m_\t$ is the number of irreducible components of the fiber at $\T=\t$. Let $\Delta$ denote the discriminant of $\overline{\scE_\s}$, and let $\val_\t(\Delta)$ denote the valuation of $\Delta$ at $\T=\t$. For each $\t\in\badred$, $m_\t$ will either equal $\val_\t(\Delta)$ (if the fiber has multiplicative reduction) or $\val_\t(\Delta)-1$ (if the fiber has additive reduction)~\cite[Equation (13)]{schuttshioda}. So if we let $\badred_a$ be the set of places with additive reduction, we obtain the formula
	\[\rank\overline{\scE_\s}(\overline{\numfield}(\T))=8-\left(\sum_{\t \in\badred} \val_\t(\Delta)\right)+\#\badred+\#\badred_a.\]
	We compute each of these terms in Appendix~\ref{proofgenrank1}; the rank will be $8-12+3+2=1$ for $\s=1$, and $8-12+5+1=2$ for $\s\neq 1$. These are upper bounds for the rank of $\scE_\s(\numfield(\T))$.
	
	We then consider the points in $\scE_\s(\overline{\numfield}(\T))$ given by
	\begin{align*}
	\caP&:=\left(1-\s -2\T ,\, 1-\s -3\T \right),\\
	\caQ&:=\left(\T ,\, (1-\s -3\T )\sqrt{\s }\right).
	\end{align*}
	If $\s=1$, then $\caP$ is non-torsion by a height computation (Appendix~\ref{proofgenrank1}) and so $\rank\scE_1(\numfield(\T))=1$ (note that $\caP=-2\caQ$). If $\s\neq 1$, on the other hand, we show that these two points are independent by computing their height pairing matrix (Appendix~\ref{proofgenrank1}), so they generate a finite-index subgroup of $\scE_\s(\overline{\numfield}(\T))$. If in addition $\s\in\numfield^2$, so $\caP$ and $\caQ$ are both defined over $\numfield$, then $\rank\scE_\s(\numfield(\T))=2$. 
	
	Now suppose $\s\notin\numfield^2$. Then $\caP$ is fixed by all Galois automorphisms, but there is an automorphism that sends $\caQ\mapsto -\caQ$. Given any $\caT\in\scE_\s(\overline{\numfield}(\T))$, we will have $\ell \caT=m\caP+n\caQ$ for some integers $\ell,m,n$ with $\ell\neq 0$. If $\caT$ (and hence $\ell \caT$) is fixed by the Galois action, then $m\caP+n\caQ=m\caP-n\caQ$, which implies $n=0$ because $\caQ$ is non-torsion. Therefore any $\caT\in\scE_\s(\numfield(\T))$ is linearly dependent with $\caP$, proving that $\scE_\s(\numfield(\T))$ has rank $1$.
\end{proof}

\section{Examples}\label{examples}

\subsection{Generating Curves in a Subfamily}

While we do not know how to rule out the possibility that $\scE_\s$ has elevated rank, we can easily compute lists of curves in this family that can be used to generate clean pairs. For example, set $\numfield=\QQ$ and consider specializations of $\scE_1$. At each  $\t=\frac{p}{q}\in\QQ$, an integral model for the fiber at $\t$ is given by
\[y^2=(x-pq)^2(x+2pq)+9p^2q^4.\]
Define the height of this curve to be
\[h(\t):=\max\{(3p^2q^2)^3,\,(2p^3q^3+9p^2q^4)^2\}.\]
Now fix some bound $H$; for each $\t\in\QQ$ with $h(\t)\leq H^6$, we check to see whether the discriminant is nonzero, and whether the point $(-2pq,-3pq^2)$ is non-torsion (guaranteeing that $(1,\t)\in\zaropen(\numfield)$). If so, we record the rank of the corresponding curve. The data is summarized in Table \ref{Eabdata}. In particular, the $27062$ rank $1$ curves found here all have $\s=1$, and so any two of them will form a clean pair.

The density conjecture~\cite[Appendix A]{conradconrad} predicts that $100\%$ of curves in this family have ranks $1$ or $2$, so the increasing proportion of rank $3$ curves in Table~\ref{Eabdata} may be concerning. However, it is likely that this trend reverses for large enough values of $H$, with the proportion of rank $3$ curves eventually decreasing to $0$.\footnote{In an analogous setting, Zagier considered all curves of the form $x^3+y^3=m$ with $m\leq 70000$, and found $38.3\%$ with rank $0$, $48.9\%$ with rank $1$, $11.7\%$ with rank $2$, and $1.1\%$ with higher rank~\cite{zagierk}; once we account for the difference in generic rank, the similarity to Table~\ref{Eabdata} is striking. However Watkins later extended the data to all $m\leq 10^7$ to show that the proportion of curves with rank $\geq 2$ appears to decay after a sufficiently long time~\cite{watkins}.}

\begin{table}
	\begin{center}
		\small
		\begin{tabular}{|c|c|c|c|c|c|c|}
			\hline
			$H$ & Total & rank $1$ & rank $2$ & rank $3$ & rank $\geq 4$ & rank ? \\
			\hline
			$10$ & 823 & 465 (56.5\%) & 339 (41.2\%) & 19 (2.3\%) & 0 (0\%) & 0 \\
			$20$ & 4710 & 2115 (44.9\%) & 2263 (48.0\%) & 332 (7.0\%) & 0 (0\%) & 0 \\
			$30$ & 13055 & 5363 (41.0\%) & 6418 (49.2\%) & 1242 (9.5\%) & 32 (0.2\%) & 0 \\
			$40$ & 26828 & 10512 (39.2\%) & 13140 (48.9\%) & 3063 (11.4\%) & 113 (0.4\%)  & 0\\
			$50$ & 46956 & 17573 (37.4\%) & 23121 (49.2\%) & 5994 (12.8\%) & 258 (0.5\%) & 10 \\
			$60$ & 74069 & 27062 (36.6\%) & 36378 (49.1\%) & 10087 (13.6\%) & 523 (0.7\%) & 19\\
			\hline
		\end{tabular}
	\end{center}
	\caption{Distribution of ranks among elliptic curves $y^2=(x-\t )^2(x+2\t)+9\t^2$ with $h(\t)\leq H^6$, such that $(-2\t,-3\t)$ is non-torsion.}
	\label{Eabdata}
\end{table}

\subsection{Curves with Small Conductor}

Consider the $683$ elliptic curves of rank $1$ with conductor up to $500$ (using Cremona's Tables~\cite{cremona}). When put into reduced Weierstrass form, $89$ of them satisfy the conditions of Corollary~\ref{infwithE} ($91$ have the form $y^2=x^3-3\t^2x+b$, and of these, there are $2$ for which $b-2\t^3=r^2$ and $(-2\t,r)$ is torsion); the first four of these have Cremona references 43a1, 65a1, 89a1, and 99a1. In particular, there are $16$ for which $b-2\t^3$ is a square,\footnote{43a1,
	112a1,
	135a1,
	153a1,
	155c1,
	216a1,
	225e1,
	236a1,
	248a1,
	252b1,
	280a1,
	304c1,
	308a1,
	364b1,
	387c1, and
	400c1.} 
so that we can take $\s=1$ for each of them; this gives us $256$ clean pairs.

The two rank $1$ curves of smallest conductor are 37a1 and 43a1. Despite 37a1 not appearing in the family $\scE$, we can use alternative techniques to prove that (37a1, 43a1) is a clean pair. Namely, pick non-torsion points $P\one,P\two$ on each, and consider the curve $C_\r:f(x\one)=\r^2g(x\two)$ passing through $\widehat{\pi}(P\one,P\two)$ as in Section~\ref{kummerrational}. This will be a genus $1$ curve, so we can use elliptic curve computations to find a principal divisor on $\overline{C_\r}$ relating $\pi(P\one,P\two)$ to the images of fixed points of $\iota$. As before, when we pull back to obtain principal divisors on $\pi^{-1}(\overline{C_\r})$ and $\pi^{-1}(\overline{C_{-\r}})$ and take their difference, the fixed points of $\iota$ will cancel, leaving us with a nonzero multiple of $\Phi\left(([P\one]-[-P\one])\otimes ([P\two]-[-P\two])\right)$. This technique (and others) will be discussed in more depth in the author's forthcoming thesis~\cite{myphd}; using these methods we can prove the cleanness of several pairs of curves that are not accounted for by Theorem~\ref{cleanthm}.\footnote{For example, of the $\binom{10}{2}=45$ pairs of rank $1$ curves with conductor below $80$, we can show that the seven pairs (37a1, 43a1), (37a1, 57a1), (37a1, 77a1), (53a1, 58a1), (61a1, 65a1), (61a1, 65a2), and (65a2, 79a1) are clean.}

However, there are still many pairs of rank $1$ curves which we have not been able to prove are clean, including for example (37a1, 53a1) and (43a1, 53a1).

\appendix
\section{Computations for Proposition~\ref{genrank1}}\label{proofgenrank1}

\begin{table}
	\begin{minipage}{12.5cm}
		\setcounter{mpfootnote}{\value{footnote}}
		\renewcommand{\thempfootnote}{\arabic{mpfootnote}}
		
		\caption{Computing the rank of $\overline{\scE_\s}(\overline{\numfield}(\T))$.}\label{rankcomps}
		\[\small\def\arraystretch{1.5}
		\begin{array}{|r|c|c|}
		\multicolumn{3}{c}{\text{Case }\s=1}\\
		\hline
		\text{model} & y^2=x^3-3\T ^2x+2\T ^3+9\T^2 & y'^2=x'^3-3\T'^2x'+2\T'^3+9\T'^4\\
		\Delta &-3888 \T^4 (9 + 4 \T) & -3888\T'^7(4 + 9\T' )\\
		\hline
		\caP & (-2\T,-3\T) & (-2\T',-3\T'^2)\\
		\hline
		\end{array}\]
		
		\[\small\def\arraystretch{1.5}
		\begin{array}{|r|c|c|c|c||c|}
		\hline
		\T & \t\notin\badred &0 & -\frac94 & \infty & \sum_{\t}\\
		\hline
		\val_\t(\Delta) & 0 & 4 & 1 & 7 & 12\\
		\text{reduction} & & y^2=x^3 &y^2=\left(x - \frac94\right)^2 \left(x + \frac92\right) & y'^2=x'^3 & \\
		\text{type} & \text{good} & \text{additive} & \text{multiplicative} & \text{additive} & \\
		\caP_\t & \text{smooth} & \text{singular} & \text{smooth} & \text{singular} & \\
		\lambda_\t(\caP) & 0 & 0\quad(\footnote{$\val_0(F_2)=2$ and $\val_0(F_3)=6$, so $\lambda_0(\caP)=-2/6+4/12$.}) & \frac1{12} & \frac1{12}\quad (\footnote{$\val_\infty(F_2)=4$ and $\val_\infty(F_3)=8$, so $\lambda_\infty(\caP)=-8/16+7/12$.}) & \frac16\\
		\hline
		\end{array}\]
		
		\begin{align*}\small\def\arraystretch{1.5}
		\begin{array}{|r|c|c|}
		\multicolumn{3}{c}{\text{Case }\s\neq 1}\\
		\hline
		\text{model} & {\scriptstyle y^2=x^3-3\T ^2x+2\T ^3+\left(1-\s-3\T\right)^2 \s} & {\scriptstyle y^2=x^3-3\T'^2x+2\T'^3+\T'^4\left(\T'-\s\T'-3\right)^2 \s}\\
		\Delta &{\scriptstyle -432\s (1-\s -3\T )^2(4 \T ^3 + (1-\s -3\T )^2 \s )} & {\scriptstyle -432\s\T'^7 (\T'-\s\T' -3 )^2(4 + (\T'-\s\T' -3)^2 \s\T' )}\\
		\hline
		\caP & \left(1-\s -2\T ,\, 1-\s -3\T \right) & \left((1-\s)\T'^2 -2\T' ,\, (1-\s)\T'^3 -3\T'^2 \right)\\
		\caQ & \left(\T ,\, (1-\s -3\T )\sqrt{\s }\right) & \left(\T' ,\, ((1-\s)\T'^3 -3\T'^2)\sqrt{\s }\right)\\
		\caP+\caQ & \hspace{-2.5cm}\big(\T-2(\sqrt{\s }-\s) , & \hspace{-2.5cm}\big(\T'-2(\sqrt{\s }-\s)\T'^2 , \vspace{-2pt}\\
		&\hspace{1cm} (3\T-(4\sqrt{\s}-3\s -1))\sqrt{\s }\big) &\hspace{1cm} (3\T'^2-(4\sqrt{\s}-3\s -1)\T'^3)\sqrt{\s }\big)\\
		\hline
		\end{array}\end{align*}
		
		\[\small\def\arraystretch{1.5}
		\begin{array}{|r|c|c|c|c||c|}
		\hline
		\t & 
		\t\notin\badred &
		\frac{1-\s}{3} & 
		r_1,r_2,r_3\quad (\footnote{the roots of $4 \t ^3 + (1-\s -3\t )^2 \s =0$. The polynomial $(1-\s -3\T)(4 \T^3 + (1-\s -3\T)^2 \s )$ in $\T$ has discriminant $6912 (\s -1)^9 \s ^2\neq 0$, so these are distinct from each other and from $\frac{1-\s}{3}$.}) & 
		\infty & 
		\sum_{\t}\\
		
		\hline
		\val_\t(\Delta) & 0 & 2 & 1 & 7 & 12\\
		
		\text{reduction} & & y^2=(x-\t)^2(x+2\t) &y^2=(x+\t)^2(x-2\t) & y^2=x^3 & \\
		
		\text{type} & \text{good} & \text{multiplicative} & \text{multiplicative} & \text{additive} & \\
		
		\caP_\t & \text{smooth} & \text{singular} & \text{smooth}\footnote{$2y\neq 0$ at all $\T\neq \frac{1-\s}{3}$.} & \text{singular} & \\
		
		\lambda_\t(\caP) & 0 & -\frac1{12}\quad(\footnote{$\val_\t(2y)=1$, so $\alpha=\frac12$.}) & 
		\addtocounter{mpfootnote}{-3}
		\frac1{12}  & \frac1{12}\quad(\footnotemark[\value{mpfootnote}]) & \frac14 \\
		\addtocounter{mpfootnote}{2}
		
		\caQ_\t & \text{smooth} & \text{singular} & \text{smooth}\footnotemark[\value{footnote}] & \text{singular} & \\
		\stepcounter{mpfootnote}
		
		\lambda_\t(\caQ) & 0 & -\frac1{12}\quad(\footnotemark[\value{mpfootnote}]) \stepcounter{mpfootnote} & \frac1{12}  & -\frac1{24}\quad(\footnote{$\val_\infty(F_2)=4$ and $\val_\infty(F_3)=10$, so $\lambda_\infty=-10/16+7/12$.}) & \frac18 \\	
		
		(\caP+\caQ)_\t & \text{smooth} & \text{smooth}\footnote{The roots of $3x^2-3\T^2$ and $2y$ are distinct when $\s\neq 1$.} &
		\text{smooth}\footnotemark[\value{mpfootnote}] & \text{singular} & \\
		
		\addtocounter{mpfootnote}{-1}
		\lambda_\t(\caP+\caQ) & 0 & \frac1{6} & \frac1{12} & -\frac1{24}\quad (\footnotemark[\value{mpfootnote}]) & \frac38\\
		\hline
		\end{array}\]
		
	\end{minipage}
\end{table}

We consider two minimal models of $\overline{\scE_\s}$: the original Weierstrass equation, and the equation obtained by the substitution $(x,y,\T)=\left(\frac{x'}{\T'^2},\frac{y'}{\T'^3},\frac{1}{\T'}\right)$ (for studying the fiber at $\infty$). We calculate the discriminant of each model; the places of bad reduction will be determined by where the discriminant vanishes. At each place, we compute the valuation of $\Delta$ and the reduction type of $\overline{\scE_\s}$. We then compute local heights of certain points using Silverman's algorithm, as described in exercises 6.7 and 6.8 of~\cite{silverman_adv}. The results of these computations are included in Table~\ref{rankcomps}, with occasional footnotes describing how the computation was done.

These computations give us the following results:
\begin{itemize}
	\item The point $\caP\in\scE_1(\numfield(\T))$ is non-torsion, because its canonical height is $\frac16$. This is used in Section~\ref{parametrization} to prove that $\scP$ is non-torsion in $\scE(\numfield(\S,\T))$, and in the proof of Proposition~\ref{genrank1} to show $\scE_1(\numfield(\T))$ has rank $1$.
	\item $\sum_{\t\in R}\val_\t(\Delta)=12$. This, together with the classification of places with bad reduction, allows us to compute the rank of $\overline{\scE_\s}(\overline{\numfield}(\T))$ in the proof of Proposition~\ref{genrank1}.
	\item The canonical heights of $\caP$, $\caQ$, and $\caP+\caQ$ on $\scE_\s$ for $\s\neq 0,1$ are $\frac14$, $\frac18$, and $\frac38$ respectively. This proves that $\caP$ and $\caQ$ are non-torsion, and since $\widehat{h}(\caP)+\widehat{h}(\caQ)=\widehat{h}(\caP+\caQ)$, they are orthogonal under the height pairing. In particular, $\caP$ and $\caQ$ are linearly independent, which is used in the proof of Proposition~\ref{genrank1} to prove that they generate a finite-index subgroup of $\overline{\scE_\s}(\overline{\numfield}(\T))$.
\end{itemize}

\pagebreak

\printbibliography

\end{document}